\documentclass[10pt,reqno]{amsart}
\usepackage[english]{babel}
\usepackage{amsmath}
\usepackage[misc]{ifsym}
\usepackage{mathrsfs,mathtools,thmtools}
\usepackage{faktor} 
\usepackage[utf8]{inputenc}
\usepackage[left=2.9cm,right=2.9cm,top=2.8cm,bottom=2.8cm]{geometry}
\usepackage{graphicx}
\usepackage{enumitem}
\usepackage{todonotes, cancel}
\usepackage[dvipsnames]{xcolor}
\usepackage[colorlinks=true,urlcolor=blue,citecolor=blue,linkcolor=blue,linktocpage,pdfpagelabels,bookmarksnumbered,bookmarksopen]{hyperref}

\usepackage{pgfplots}
\pgfplotsset{compat=1.18}

\usepackage{type1cm}

\newtheorem{theorem}{Theorem}[section]
\newtheorem{lemma}{Lemma}[section]
\newtheorem{prop}[lemma]{Proposition}

\theoremstyle{remark}
\newtheorem{rmk}[lemma]{Remark}

\newcommand{\R}{\mathbb R}
\newcommand{\N}{\mathbb N}
\newcommand{\dmu}{\, {\rm d} \mu}
\newcommand{\dnu}{\, {\rm d} \nu}
\newcommand{\dz}{{\mathrm{d}z}}

\allowdisplaybreaks 
\numberwithin{equation}{section}

\begin{document}
\title[Infinitely many solutions for a critical $p$-Grushin problem]{A note on critical problems involving the $p$-Grushin Operator: existence of infinitely many solutions}

\author[P. Malanchini]{Paolo Malanchini}
\address[Paolo Malanchini]{Dipartimento di Matematica e Applicazioni\\ Universit\`a degli Studi di Milano - Bicocca, via Roberto Cozzi 55, 20125 - Milano, Italy}
\email{p.malanchini@campus.unimib.it}

\author[G. Molica Bisci]{Giovanni Molica Bisci}
\address[Giovanni Molica Bisci]{Department of Human Sciences and Promotion of Quality of Life, San Raffaele University, via di Val Cannuta 247, I-00166 Roma, Italy} 
\email{giovanni.molicabisci@uniroma5.it}

\author[S. Secchi]{Simone Secchi}
\address[Simone Secchi]{Dipartimento di Matematica e Applicazioni\\ Universit\`a degli Studi di Milano - Bicocca, via Roberto Cozzi 55, 20125 - Milano, Italy}
\email{simone.secchi@unimib.it}

\keywords{$p$-Grushin operator, critical exponent, genus} \subjclass[2020]{35J70, 35B33, 35J20}

\begin{abstract}
We consider a critical problem in a bounded domain involving the $p$-Grushin operator $\Delta_\alpha^p$.
After a truncation argument, we obtain infinitely many solutions to our problem via Krasnoselskii's genus, extending a previous result of Garc\'ia Azorero and Peral Alonso in \cite{GP91} to the $p$-Grushin operator.
A central part of our analysis is the verification of the Palais-Smale condition of the associated functional under a certain level.
\end{abstract}

\maketitle

\section{Introduction}
In this note we consider the boundary-value problem
\begin{equation} 	
\label{problem:main}
\left\{ 
    \begin{alignedat}{2}
       - \Delta_\alpha^p u & = \lambda \lvert u \rvert^{q-2} u + \lvert u \rvert^{p^*_\alpha -2}u \quad &&\mbox{in} \;\; \Omega,\\
		    u &= 0 \quad &&\mbox{on} \;\; \partial\Omega, 
    \end{alignedat} 
\right.     
\end{equation}
where $\Omega$ is a bounded domain in $\R^N$, $\lambda>0$, $q\in (1,p)$, $\Delta_\alpha^p$ is the $p$-Grushin operator and $p^*_\alpha$ is the Sobolev critical exponent in this framework defined in \eqref{pstar}.

For the reader's convenience, We briefly recall the definition of the $p$-Grushin operator $\Delta_\alpha^p$. If we split the euclidean space $\R^N$ as $\R^m \times \R^\ell$, where $m \geq 1$, $\ell \geq 1$ satisfy $m+\ell=N$, a generic point $z \in \R^N$ can be written as
\begin{displaymath}
	z = \left(x,y\right) = \left( x_1,\ldots,x_m,y_1,\ldots,y_\ell \right),
\end{displaymath}
where $x \in \R^m$ and $y \in \R^\ell$.
We fix a nonnegative real parameter $\alpha$, and we define the Grushin gradient $\nabla_\alpha = (X_1,\dots, X_N)$ as the system of vector fields
\begin{displaymath}
X_i = \frac{\partial}{\partial x_i}, ~ i=1,\dots, m,\quad
 X_{m+j} = \lvert x \rvert^\alpha \frac{\partial}{\partial y_j}, ~ j=1,\dots,\ell.
\end{displaymath}
The positive integer \[N_\alpha \coloneqq m + (1+\alpha)\ell\]
is the \emph{homogeneous dimension} associated to the decomposition $N=m+\ell$.

For each $p\in (1,N_\alpha)$ and for each differentiable function $u\colon\R^N\to\R$, we define the \emph{Grushin $p$-Laplace Operator} ($p$-Grushin for short) by
\begin{equation*}
	\Delta_\alpha^p u \coloneqq \sum_{i=1}^N X_i (|\nabla_\alpha u|^{p-2} X_iu).
\end{equation*}
If $\alpha=0$ the operator $\Delta_\alpha^p$ reduces to the $p$-Laplace operator $\Delta_p$, while for $p=2$ it coincides with the Baouendi--Grushin operator 
\begin{equation*}
    \Delta_\alpha =  \Delta_x + |x|^{2\alpha}\Delta_y,
\end{equation*}
where $\Delta_x$ and $\Delta_y$ are the Laplace operators in the variables $x$ and $y$, respectively.
This degenerate operator was introduced by Baouendi \cite{baouendi} and by Grushin \cite{grushin}.
A crucial property of the Baouendi-Grushin operator is that it is not uniformly elliptic in $\R^N$, since it is degenerate on the subspace $\Sigma=\{0\} \times \R^\ell$. Going back to our problem \eqref{problem:main}, when $\Omega$ is disjoint from $\Sigma$, the Grushin operator reduces to a uniformly elliptic operator, and its theory is classical.
For this reason, when dealing with the ($p$-)Grushin operator, it is usually assumed that
\begin{equation}
\label{eq_cond_omega}
    \Omega \cap \Sigma \ne \emptyset.
\end{equation}

\medskip
It seems that Bieske and Gong in \cite{B06} were the first to treat a $p$-Laplace equation in Grushin-type spaces, while fundamental solutions for the $p$-Laplace equation on a class of Grushin vector spaces were studied by Bieske in \cite{B11}.

More recent works, such as those by Huang and Yang \cite{HY23} and Huang, Ma, and Wang \cite{HMW15}, investigate the existence, uniqueness, and regularity of solutions to the $p$-Laplace equation involving Grushin-type operators. Moreover, in \cite{WCCY20}, the authors prove a Liouville-type theorem for stable solutions of weighted $p$-Grushin equations. In \cite{AYZ25} a sharp remainder formula for the Poincar\'e inequality associated with $p$-Grushin vector fields was established. Finally, in the very recent \cite{CC26} existence and positivity of solutions for critical problems involving the $p$-Grushin operator with a subcritial perturbation are studied.

\medskip
The study of critical problems started with the celebrated paper~\cite{BN83}, where Br\'ezis and Nirenberg proved the existence of positive solutions for the problem
\begin{equation*} 
\left\{ 
    \begin{alignedat}{2}
       - \Delta u & = \lambda u + \lvert u \rvert^{2^* -2}u \quad &&\mbox{in} \;\; \Omega,\\
		    u &= 0 \quad &&\mbox{on} \;\; \partial\Omega, 
    \end{alignedat} 
\right.     
\end{equation*}
 on a bounded domain $\Omega\subset \mathbb{R}^N$, depending on the values of the positive parameter $\lambda$.
 Here $N \geq 3$ and $2^*=2N/(N-2)$ is the critical Sobolev exponent.

Analogous results for the $p$-Laplace operator have been obtained by Garc\'ia Azorero and Peral Alonso in \cite{GP91}. They considered the problem
\begin{equation} 	
\label{problem:peral}
\left\{ 
    \begin{alignedat}{2}
       - \Delta_p u & = \lambda \lvert u \rvert^{q-2} u + \lvert u \rvert^{p^* -2}u \quad &&\mbox{in} \;\; \Omega,\\
		    u &= 0 \quad &&\mbox{on} \;\; \partial\Omega, 
    \end{alignedat} 
\right.     
\end{equation}
in a smooth bounded domain $\Omega\subset \mathbb{R}^N$ with $p^* = Np/(N-p)$.
We can gather their main results as follows: see Theorem 3.2, Theorem 3.3, and Theorem 4.5 in \cite{GP91}.
\begin{theorem}\label{thm_peral}
    \begin{enumerate}[label=(\roman*)]
        \item If $q\in (p,p^*)$, there exists $\lambda_0>0$ such that \eqref{problem:peral} has a nontrivial solution for all $\lambda\ge \lambda_0$.
        \item If $q\in \left( \max\{p, p^*- p/(p-1)\}, p^* \right)$, there exists a nontrivial solution of \eqref{problem:peral} for all $\lambda>0$.
        \item If $q\in (1,p)$, there exists $\bar\lambda>0$ such that, for $0<\lambda<\bar\lambda$, problem \eqref{problem:peral} admits infinitely many solutions.
    \end{enumerate}
\end{theorem}
Problem~\eqref{problem:main} with $p = 2$ was investigated in \cite{AGLT24}, where the authors extend the Br\'ezis--Nirenberg result in the Grushin setting. 
Critical problems associated with the Grushin operator have been extensively studied. See also \cite{GANDAL2026130363} for a multiplicity result in the case $p=2$. 
We refer, among others, to the recent papers \cite{KMB25,L25}, which address existence and multiplicity results for critical Grushin problems with a Hardy term and for the Grushin--Choquard equation, respectively.

A multiplicity result for problem \eqref{problem:main} with $q=p$ was proved in \cite{MMBS25}, extending to all $\alpha>0$ the result of \cite{PSY16} for the $p$-Laplace operator.
Gandal, Loiudice and Tyagi recently studied problem \eqref{problem:main} with $p\in (q,p^*_\alpha)$ in \cite{GLT25}, extending the result of Garc\'ia Azorero and Peral Alonso to the $p$-Grushin framework. Among other results, the authors established the counterparts of statements~(i) and~(ii) in Theorem~\ref{thm_peral} for problem~\eqref{problem:main}.

\bigskip

In this paper we prove the existence of infinitely many solutions to \eqref{problem:main}, extending Theorem \ref{thm_peral} (iii) to the $p$-Grushin framework and concluding the study of critical problems for the $p$-Grushin operator started in \cite{GLT25}.

\medskip
We will look for solutions of \eqref{problem:main} in the Sobolev space $\mathring{W}^{1,p}_\alpha(\Omega)$  defined as the completion of $C^1_c(\Omega)$ with respect to the norm 
\begin{displaymath}
\|u\|_{\alpha,p} = \left(\int_\Omega \lvert \nabla_\alpha u \rvert^{p}\,\dz\right)^{1/p},
\end{displaymath}
whose details are given in Section \ref{subsec_func}.

We say that a function $u\in \mathring{W}^{1,p}_\alpha(\Omega)$ is a weak solution of \eqref{problem:main} if\footnote{The symbol $\nabla_\alpha u \nabla_\alpha \varphi$ is a shorthand for $\nabla_\alpha u \cdot \nabla_\alpha \varphi$.}
\begin{equation*}
    \int_\Omega |\nabla_\alpha u|^{p-2} \nabla_\alpha u \nabla_\alpha \varphi\,\dz =\lambda \int_\Omega |u|^{q-2}u \varphi\,\dz  + \int_\Omega |u|^{p^*_\alpha-2} u \varphi\,\dz,
\end{equation*}
for all $\varphi\in\mathring{W}^{1,p}_\alpha(\Omega),$
and so weak solutions of \eqref{problem:main} coincide with critical points of the functional $I_\lambda\colon\mathring{W}^{1,p}_\alpha(\Omega)\to\R$ defined by
\begin{equation} \label{eq_def_functional}
    I_\lambda(u) \coloneqq \frac{1}{p} \int_\Omega |\nabla_\alpha u|^p\,\dz - \frac{\lambda}{q} \int_\Omega |u|^q\,\dz - \frac{1}{p^*_\alpha}\int_\Omega |u|^{p^*_\alpha}\,\dz.
\end{equation}
We can state our main result.
\begin{theorem}
    \label{thm_main}
    Let $q\in (1,p)$ and let $\Omega\subset\R^N$ be a smooth bounded domain satisfying \eqref{eq_cond_omega}. There exists $\bar\lambda>0$ such that for all $\lambda\in (0,\bar\lambda)$, problem \eqref{problem:main} has infinitely many weak solutions in $\mathring{W}^{1,p}_\alpha(\Omega)$.
\end{theorem}
The paper is organized as follows. In Section~\ref{sec_prelim}, we introduce the functional framework and recall some preliminary results, including a concentration--compactness theorem, the main properties of Krasnoselskii's genus, and a classical deformation lemma.
In Section~\ref{sec_PS}, we establish the Palais-Smale condition for the functional $I_\lambda$ under a suitable energy level. Finally, in Section~\ref{sec_genus}, we prove the existence of infinitely many solutions to problem~\eqref{problem:main}, thus completing the proof of Theorem~\ref{thm_main}.

\section{Preliminaries}
\label{sec_prelim}

\subsection{Functional setting}
\label{subsec_func}
Fix a bounded domain $\Omega\subset\mathbb{R}^N$ and a number $p\in (1,\infty)$. The Sobolev space $\mathring{W}^{1,p}_\alpha(\Omega)$ is defined as the completion of $C_c^1(\Omega)$ with respect to the norm
\begin{displaymath}
\|u\|_{\alpha,p} = \left(\int_\Omega \lvert \nabla_\alpha u \rvert^{p}\,\dz\right)^{1/p}.
\end{displaymath}
Specially, the space $\mathring{W}^{1,2}_\alpha(\Omega) = \mathscr{H}_\alpha(\Omega)$ is a Hilbert space endowed with the inner product
\begin{displaymath}
	\langle u,v\rangle_\alpha = \int_{\Omega}^{} \nabla_\alpha u \nabla_\alpha v\, \, \dz.
\end{displaymath}
The following embedding result was proved in \cite[Proposition 3.2 and Theorem 3.3]{KL12}, see also \cite[Corollary 2.11]{HY23}.
\begin{prop} \label{prop_sobolev}
	Let $\Omega\subset\mathbb{R}^N$ be a bounded open set. Then the embedding
	\begin{displaymath}
		\mathring{W}^{1,p}_\alpha(\Omega) \hookrightarrow L^r(\Omega)
	\end{displaymath}
	is continuous for every $r\in [1,p^*_\alpha]$ and compact for every $r\in [1,p^*_\alpha)$,
    where $p^*_\alpha$ is the critical Sobolev exponent for the $p$\nobreakdash-Grushin operator defined as
	\begin{equation}
		\label{pstar}
		p^*_\alpha = \frac{p N_\alpha}{N_\alpha - p}.
	\end{equation}
\end{prop}
We may thus define the best constant of the Sobolev embedding $\mathring{W}^{1,p}_\alpha(\Omega)\hookrightarrow L^{p^*_\alpha}(\Omega)$:
\begin{equation}
	\label{eq:sobconst}
	S \coloneqq \inf_{\substack{ u \in \mathring{W}^{1,p}_\alpha(\Omega) \\ u \neq 0}} \frac{\int_\Omega \left\vert \nabla_\alpha u \right\vert^p \, \dz}{\left( \int_\Omega \left\vert u \right\vert^{p_\alpha^*} \, \dz \right)^{p/p_\alpha^*}}.
\end{equation}
For $z  =(x,y)\in\R^N = \R^m\times \R^\ell$ let
\begin{equation*}
d(z) = \left( |x|^{2\left(\alpha +1\right)} + (\alpha+1)^2 |y|^2\right)^{\frac{1}{2\left(\alpha+1\right)}},
\end{equation*}
be the homogeneous distance associated with the Grushin geometry. We indicate with \[B_R(z) =\{\xi\in\R^N :\,  d(z -\xi)<R \}\] the $d$-ball of center $z\in\R^N$ and radius $R>0$. A direct computation shows that $|B_R(z)| = |B_1(z)| R^{N_\alpha}$, where $|A|$ denotes the Lebesgue measure of $A\subset\R^N$.

\subsection{Some tools}\label{sebsection_tools}

We will need a concentration-compactness result for the $p$-Grushin operator proved in \cite[Theorem 3.3]{MMBS25}.
The space of all finite signed Radon measures on $\Omega\subset\R^N$ is denoted by $\mathcal{M}(\Omega,\R)$. 
A sequence of measures $(\mu_n)\subset \mathcal{M}(\Omega,\R)$ converges tightly to a measure $\mu$, written as $\mu_n \stackrel{*}{\rightharpoonup} \mu$, if
\begin{equation}
\label{measconv}
\int_{\Omega} f \, {\rm d}\mu_n \to \int_{\Omega} f \, {\rm d}\mu \quad \mbox{for all} \;\; f\in C_b(\Omega),
\end{equation}
where $C_b(\Omega)$ is the space of the bounded, continuous functions on $\Omega$. On the other hand, $(\mu_n)\subset \mathcal{M}(\Omega,\R)$ is said to converge weakly to $\mu$, written as $\mu_n \rightharpoonup \mu$, if \eqref{measconv} holds for all $f\in C_0(\Omega)$, where $C_0(\Omega)$ is the space of the continuous functions that vanish at infinity.

\begin{theorem}[{\cite[Theorem~3.3]{MMBS25}}]
	\label{thm:concentration}
    Suppose that $(u_n)\subset \mathring{W}^{1,p}_\alpha(\Omega)$ satisfies $u_n\rightharpoonup u$ in $\mathring{W}^{1,p}_\alpha(\Omega)$, and $| \nabla_\alpha u_n|\rightharpoonup \mu$, $| u_n |^{p^*_\alpha} \stackrel{*}{\rightharpoonup}\nu$ for some $u\in \mathring{W}^{1,p}_\alpha(\Omega)$ and $\mu$, $\nu$ are bounded non-negative measures on $\Omega$.
 	Then there exist an at most countable set $\mathcal{A}$, a family $\lbrace z_j\rbrace_{j\in \mathcal{A}}$ of distinct points in $\Omega$ and two families of positive numbers $\lbrace \mu_j \rbrace_{j\in \mathcal{A}}$, $\lbrace \nu_j \rbrace_{j\in \mathcal{A}}$ such that
		\begin{equation*}
		   \nu = | u |^{p^*_\alpha} + \sum_{j\in \mathcal{A}} \nu_j \delta_{z_j},  \quad  
              \mu\ge | \nabla_\alpha u|^p + \sum_{j\in \mathcal{A}}\mu_j \delta_{z_j}, \quad
              \mu_j \ge S \nu_j^{\frac{p}{p^*_\alpha}}\quad\hbox{for all $j\in \mathcal{A}$,} 
		\end{equation*}
        where $\delta_z$ indicates the Dirac delta concentrated at $z\in\R^N$.
\end{theorem}

We now recall the notion of Krasnoselskii's genus. Let $X$ be a Banach space and $\mathcal{E}$ be the class of the closed and symmetric with respect to the origin subsets of $X\setminus\{0\}$.
 
The genus of $A\in\mathcal{E}$ is the smallest positive integer $k$ such that there exists an odd and continuous map $\phi\colon A\rightarrow \mathbb{R}^k\setminus \{0\}$ and it is denoted by $\gamma (A)$. If there is no such map, we set $\gamma(A)=+\infty$. Finally, we define $\gamma(\emptyset)=0$.

The main properties of the genus are collected in the following proposition, see e.g. \cite[Lemma 1.2]{AR73}. For $A\in\mathcal{E}$ and $\delta>0$, let $N_\delta(A)$ denote a uniform $\delta$-neighborhood of $A$, i.e. 
\begin{displaymath}
N_\delta(A) = \{x\in X: \operatorname{dist}(x,A)\leq \delta \}.
\end{displaymath}
\begin{prop}
\label{prop_genus}
    Let $A$, $B\in \mathcal{E}$. 
    \begin{enumerate}[label=(\roman*)]
    \item If $A\subset B$, then $\gamma(A)\leq \gamma(B).$
    \item If there exists an odd homeomorphism between $A$ and $B$, then $\gamma(A)= \gamma(B).$
    \item $\gamma(A\cup B)\leq \gamma(A)+\gamma(B)$.
    \item If $\gamma(B)<\infty$, $\gamma(\overline{A\setminus B})\geq \gamma(A)- \gamma(B)$ .
    \item $\gamma(\mathbb{S}^{k-1}) = k$, where $\mathbb{S}^{k-1}$ is the $k$ dimensional sphere  in $\R^k$.
    \item If $A$ is compact then $\gamma(A)<\infty$ and there exists $\delta>0$ such that  $N_\delta(A)\in\mathcal{E}$ and $\gamma(A)=\gamma(N_\delta(A))$.
    \end{enumerate}
\end{prop}
 We will need the following deformation lemma, see \cite[Lemma 1.3]{AR73}. Given $I\in C^1(X,\mathbb{R})$, for any $c,d\in \mathbb{R}$, we denote
 \begin{equation*}
  K_c=\{u\in X: I'(u)=0~\text{and}~I(u)=c\}~\text{and}~I^d=\{u\in X:I(u)\leq d\}. 
 \end{equation*}
\begin{lemma}
\label{lemma_deformation}
Suppose $I\in C^1(X,\mathbb{R})$ satisfies the Palais-Smale condition at level $c\in \R$ and let $U$ be any neighborhood of $K_c$. Then there exist $\eta_t(x)=\eta(t,x)\in C([0,1] \times X, X )$ and constants $\varepsilon_0>\varepsilon>0$  such that 
\begin{enumerate}[label=(\roman*)]
    \item $\eta_0(x) = x$ for all $x\in X$. 
    \item  $\eta_t(x)=x$ for $x\notin I^{-1}([c-\varepsilon_0,c+\varepsilon_0])$ and all $t\in [0,1]$.
    \item $\eta_t$ is a homeomorphism of $X$ into $X$ for all $t\in [0,1]$.
    \item $I(\eta_t(x))\leq I(x)$ for all $x\in X$ and $t\in [0,1]$. 
    \item $\eta_1(I^{c+\varepsilon}\setminus  U)\subset I^{c-\varepsilon}$.
    \item If $K_c=\emptyset$, then $\eta_1(I^{c+\varepsilon})\subset I^{c-\varepsilon}$.
    \item If $I$ is even, then $\eta_t$ is odd for any $t\in [0,1]$.
\end{enumerate}
\end{lemma}

\begin{rmk}
\label{remark_epsilon}
    By \cite[Remark 1.4]{AR73}, if $c<0$ we may choose $\varepsilon_0< -c$. 
\end{rmk}

\section{Palais-Smale condition}
\label{sec_PS}
In order to prove Theorem \ref{thm_main}, we need the functional $I_\lambda$ defined in \eqref{eq_def_functional} to satisfy the Palais-Smale condition under a certain level.


\begin{lemma}
	\label{lemma:Palais-Smale}
	 $I_\lambda$ satisfies the Palais-Smale condition $(PS)_c$ for all $c<S^{N_\alpha/p} /N_\alpha - h \lambda^{\frac{p^*_\alpha}{p^*_\alpha-q}}$, where 
    \begin{equation}
    \label{eq_def_h}
h = N_\alpha^{\frac{q}{p^*_\alpha-q}} |\Omega |\left(\frac{1}{q} -\frac{1}{p} \right)^{\frac{p^*_\alpha}{p^*_\alpha-q}} \left[\left(\frac{q}{p^*_\alpha} \right)^{\frac{q}{p^*_\alpha-q}}  -\left(\frac{q}{p^*_\alpha} \right)^{\frac{p^*_\alpha}{p^*_\alpha-q}}  \right] \ge 0.
\end{equation}
\end{lemma}
\begin{proof}
	Pick any $c<S^{N_\alpha/p} /N_\alpha - h \lambda^{\frac{p^*_\alpha}{p^*_\alpha-q}}$. Consider a sequence $(u_n) \subset \mathring{W}^{1,p}_\alpha(\Omega)$ such that (here and in the following, $o(1)$ stands for a term that vanishes as $n \to \infty$) 
    \begin{equation}
    \label{eq_ps_1}
        I_\lambda(u_n) = \frac{1}{p}\|u_n\|_{\alpha,p}^p - \frac{\lambda}{q}\|u_n\|_q^q - \frac{1}{p^*_\alpha}\|u_n\|_{p^*_\alpha}^{p^*_\alpha}   = c  + o(1) 
    \end{equation}
and
    \begin{equation}
        \label{eq_ps_2}
            I'_\lambda(u_n)= o(1) \quad \hbox{in $W^{-1,p'}_\alpha(\Omega)\coloneqq \left(\mathring{W}^{1,p}_\alpha(\Omega) \right)^*$}.
    \end{equation}
First of all, we claim that
	\begin{equation}
		\label{claim:1}
		\text{the sequence $( u_n )$ is bounded in $\mathring{W}^{1,p}_\alpha(\Omega).$}
	\end{equation}
By \eqref{eq_ps_2} it follows that 
	\begin{displaymath}
        \|u_n\|_{\alpha,p}^p -\lambda\|u_n\|_q^q- \|u_n\|_{p^*_\alpha}^{p^*_\alpha} = o(1) \|u_n\|_{\alpha,p} 
	\end{displaymath}
and inserting this identity into \eqref{eq_ps_1} gives
\begin{equation*}
    \frac{1}{p} \|u_n\|_{\alpha,p}^p - \frac{\lambda}{q}\|u_n\|_q^q - \frac{1}{p^*_\alpha} \|u_n\|_{\alpha,p}^p + \frac{\lambda}{p^*_\alpha}\|u_n\|_q^q = c + o(1) + o(1)\|u_n\|_{\alpha,p},
\end{equation*}
that is
\begin{equation}
\label{eq_bound_norm}
    \frac{1}{N_\alpha}\|u_n\|_{\alpha,p}^p = o(1) \|u_n\|_{\alpha,p} + \lambda \left(\frac{1}{q} - \frac{1}{p^*_\alpha} \right) \|u_n\|_q^q + c + o(1).
\end{equation}
Observe that $\frac{1}{q} - \frac{1}{p^*_\alpha}>0$ since $q<p$. Moreover, by the continuous embedding $\mathring{W}^{1,p}_\alpha(\Omega)\hookrightarrow L^q(\Omega)$ of Proposition \ref{prop_sobolev} and \eqref{eq_bound_norm} we get
\begin{equation*}
     \|u_n\|_{\alpha,p}^p \le c_1 + c_2\|u_n\|_{\alpha,p} + c_3 \|u_n\|_{\alpha,p}^q,
\end{equation*}
where $c_1$, $c_2$ and $c_3$ are positive constants not depending on $n$. Therefore, since $q<p$, (\ref{claim:1}) is proved.

So, there exist a sequence, still denoted by $( u_n )$ and $u\in \mathring{W}^{1,p}_\alpha(\Omega)$ such that
    \begin{align}
        u_n \rightharpoonup u &~\hbox{in $\mathring{W}^{1,p}_\alpha(\Omega)$}\notag \\
        u_n \to u &~\hbox{in $L^r(\Omega)$, for any $r\in [1,p^*_\alpha)$}     \label{eq_convergence}\\
        u_n \to u &~\hbox{a.e. in $\Omega$.} \notag
    \end{align}
 Now, by Theorem \ref{thm:concentration} there exist two non-negative bounded measures $\mu,\nu$ in $\Omega$, an at most countable family of points $\lbrace z_j\rbrace_{j\in\mathcal{A}}$ and positive numbers $\lbrace \mu_j \rbrace_{j\in \mathcal{A}}$, $\lbrace \nu_j \rbrace_{j\in \mathcal{A}}$ such that
	\begin{equation}
    \label{eq_convergence_measure}
	 \lvert \nabla_\alpha u_n\rvert^p \buildrel\ast\over\rightharpoonup \mu, \quad \lvert u_n\rvert^{p^*_\alpha} \buildrel\ast\over\rightharpoonup \nu
	\end{equation}
	and
	\begin{equation}
		\label{eq_nu_mu} \nu = \lvert u \rvert^{p^*_\alpha} + \sum_{j\in\mathcal{A}} \nu_j \delta_{z_j},
		\quad\mu\ge \lvert \nabla_\alpha u\rvert^p + S\sum_{j\in\mathcal{A}}\mu_j \delta_{z_j}.
	\end{equation}
Pick $z_j\in\Omega$ for some fixed $j\in\mathcal{A}$. Let us consider a function that is $\varphi\in C^{\infty}_0(\R^N)$ such that
\begin{equation}
    \label{eq_test_phi}
    \varphi\equiv 1 ~\hbox{on $B_\varepsilon(z_j)$}, \quad \varphi\equiv 0 ~\hbox{on $\R^N\setminus B_{2\varepsilon}(z_j)$}, \quad |\nabla_\alpha\varphi|\le \frac{d}{\varepsilon}, 
\end{equation}
for a positive constant $d$. The existence of a cut-off function for general Lipschitz vector fields, as observed in \cite[p. 1850]{GL03}, was proved in \cite{FSS97, GN98}.
Since the sequence $(u_n\varphi)$ is bounded in $\mathring{W}^{1,p}_\alpha(\Omega)$, by \eqref{eq_ps_2} we get that $\langle I'_\lambda(u_n), \varphi u_n\rangle = 0$. So, passing to the limit as $n\to\infty$ by \eqref{eq_convergence} and \eqref{eq_convergence_measure}
\begin{equation}
    \label{eq_limit}
    \int_\Omega \varphi \dmu+ \lim_{n\to\infty} \int_\Omega u_n |\nabla_\alpha u_n|^{p-2} \nabla_\alpha u_n\nabla_\alpha\varphi\,\dz - \lambda\int_\Omega |u|^q \varphi\,\dz - \int_\Omega \varphi\dnu = 0.
\end{equation}
Now, by \eqref{eq_test_phi} and \eqref{eq_convergence}-\eqref{eq_convergence_measure}-\eqref{eq_nu_mu}
\begin{align*}
   0&\le \lim_{n\to\infty} \left\lvert\int_\Omega u_n |\nabla_\alpha u_n|^{p-2} \nabla_\alpha   u_n \nabla_\alpha\varphi\,\dz\right\rvert \le \lim_{n\to\infty}\int_\Omega |u_n| |\nabla_\alpha u_n|^{p-1}|\nabla_\alpha\varphi|\,\dz\\ &\le \frac{d}{\varepsilon}\lim_{n\to\infty}\int_{B_{2\varepsilon}(z_j)} |u_n| |\nabla_\alpha u_n|^{p-1}\,\dz \le \lim_{n\to\infty}\frac{d}{\varepsilon} \left(\int_{B_{2\varepsilon}(z_j)}|\nabla_\alpha u_n|^p \,\dz\right)^{\frac{p-1}{p}} \left(\int_{B_{2\varepsilon}(z_j)} |u_n|^p \,\dz\right)^{\frac{1}{p}} \\&
   = \frac{d}{\varepsilon} \mu(B_{2\varepsilon}(z_j))^{\frac{p-1}{p}} \left(\int_{B_{2\varepsilon}(z_j)} |u|^p \,\dz\right)^{\frac{1}{p}} \le \frac{d}{\varepsilon} \mu(B_{2\varepsilon}(z_j))^{\frac{p-1}{p}}  \left(\int_{B_{2\varepsilon}(z_j)} |u|^{p^*_\alpha}\,\dz \right)^{\frac{1}{p^*_\alpha}} |B_{2\varepsilon}(z_j)|^{\frac{1}{N_\alpha}} \\&\le \tilde{c} \mu(B_{2\varepsilon}(z_j))^{\frac{p-1}{p}} \left(\int_{B_{2\varepsilon}(z_j)} |u|^{p^*_\alpha}\,\dz \right)^{\frac{1}{p^*_\alpha}} \to 0\quad \text{as}~\varepsilon\to 0,
\end{align*}
where $\tilde{c}$ is a constant not depending on $\varepsilon$, since $|B_{2\varepsilon}(z_j)|\sim \varepsilon^{N_\alpha}$.

Then by \eqref{eq_limit}
\begin{equation}
\label{eq_uguale}
    0 =\lim_{\varepsilon\to 0}\left( \int_\Omega \varphi\dmu - \lambda\int_\Omega |u|^q\varphi\,\dz - \int_\Omega \varphi\dmu\right) = \mu_j - \nu_j.
\end{equation}
By Theorem \ref{thm:concentration}, $\mu_j\ge S \nu_j^{\frac{p}{p^*_\alpha}}$, and by \eqref{eq_uguale} $\nu_j\ge S \nu_j^{\frac{p}{p^*_\alpha}}$. That is either $\nu_j = 0$, or 
\begin{equation}
\label{eq_nuk}
    \nu_j \ge S^{\frac{1}{1-\frac{p}{p^*_\alpha}}} = S^{N_\alpha/p}.
\end{equation}
Note that, in particular, we can conclude that the family $\mathcal{A}$ is bounded, since $\nu$ is a finite measure. The aim now is to prove that \eqref{eq_nuk} is not admissible.

Let us assume that there exists $j_0\in\mathcal{A}$ such that $\nu_{j_0}\ne 0$, i.e. $\nu_{j_0}\ge S^{N_\alpha/p}$. 
Now
\begin{align*}
    c &= \lim_{n\to\infty} I_\lambda(u_n) = \lim_{n\to\infty}\left(I_\lambda(u_n) - \frac{1}{p} \langle I'_\lambda(u_n), u_n\rangle \right) \\
    &= \lambda\left(\frac{1}{p}-\frac{1}{q} \right)\int_\Omega |u|^q\,\dz +\frac{1}{N_\alpha}\int_\Omega |u|^{p^*_\alpha}\,\dz +\frac{1}{N_\alpha}\sum_{j\in\mathcal{A}} \nu_j\\
    & \ge \lambda\left(\frac{1}{p}-\frac{1}{q} \right)\int_\Omega |u|^q\,\dz +\frac{1}{N_\alpha}\int_\Omega |u|^{p^*_\alpha}\,\dz +\frac{1}{N_\alpha} S^{N_\alpha/p}
\end{align*}
and by H\"{o}lder's inequality
\begin{equation}
    \label{eq_est_c}
    c\ge \frac{1}{N_\alpha} S^{N_\alpha/p} + \frac{1}{N_\alpha} \int_\Omega |u|^{p^*_\alpha}\,\dz - \lambda\left( \frac{1}{q} - \frac{1}{p} \right) |\Omega|^{(p^*_\alpha-q)/p^*_\alpha} \left(\int_\Omega |u|^{p^*_\alpha}\,\dz \right)^{q/p^*_\alpha}.
\end{equation}
Now, Let us consider the function $f\colon\R\to\R$, defined as $f(x)\coloneqq A x^{p^*_\alpha} - \lambda B x^q$, with $A= 1/N_\alpha$ and $B= \left( \frac{1}{q} - \frac{1}{p} \right) |\Omega|^{(p^*_\alpha-q)/p^*_\alpha}$. This function attains its absolute minimum (for $x>0$) at the point $x_0 = (\lambda B q/p^*_\alpha A)^{\frac{1}{p^*_\alpha-q}}$. That is, by a straightforward calculation,
\begin{equation*}
f(x)\ge f(x_0) = - h  \lambda^{\frac{p^*_\alpha}{p^*_\alpha-q}},    
\end{equation*}
where $h$ is given in \eqref{eq_def_h}. By \eqref{eq_est_c}, $c\ge S^{N_\alpha/p}/N_\alpha - h \lambda^{\frac{p^*_\alpha}{p^*_\alpha-q}}$, which is in contradiction with the hypothesis on $c$.

This implies that $\nu_j = 0$ for all $j\in\mathcal{A}$, and so $\lim_{n\to\infty}\int_\Omega |u_n|^{p^*_\alpha}\,\dz = \int_\Omega |u|^{p^*_\alpha}\,\dz$,
and thus
\begin{equation}
\label{eq_conv_strong}
u_n\to u~\text{strongly in}~ L^{p^*_\alpha}(\Omega).
\end{equation}	
Now, in order to conclude with the strong convergence of $u_n$ to $u$ in $\mathring{W}^{1,p}_\alpha(\Omega)$ we proceed as in \cite[Lemma 6.1]{MMBS25}. Let $A_p\colon \mathring{W}^{1,p}_\alpha(\Omega)\to W_\alpha^{-1,p'}(\Omega)$ be the continuous operator defined  by
 \begin{displaymath}
 \langle A_p(u),v\rangle \coloneqq \int_{\Omega} \lvert \nabla_\alpha u \rvert^{p-2}\nabla_\alpha u \nabla_\alpha v \,\dz.
 \end{displaymath}
A straightforward computation shows that the sequence $( A_p(u_n))$ is a Cauchy sequence in $W_\alpha^{-1,p'}(\Omega)$. In fact $A_p(u_n) = I'_\lambda(u_n) +\lambda \lvert u_n \rvert^{q-2} u_n + \lvert u_n \rvert^{p^*_\alpha -2} u_n$ and the claim follows from $\eqref{eq_conv_strong}$ and the Palais-Smale condition.
Now, a direct application of \cite[Eq (2.2)]{simon} implies
\begin{displaymath}
	\langle A_p(u_n) - A_p(u_m), u_n - u_m 	\rangle \ge
\begin{cases}
	c_1 \|u_n - u_m\|_{\alpha,p}^p &\hbox{if $p>2$},\\
	c_2 M^{p-2} \|u_n-u_m\|_{\alpha,p}^2 &\hbox{if $1<p\le 2$},
\end{cases}
\end{displaymath}
where $c_1 = c_1(N,\alpha,p)$, $c_2 = c_2(N,\alpha,p,\Omega)$ and $M = \max\lbrace \|u_n\|_{\alpha,p}, \|u_m\|_{\alpha,p}\rbrace$.
So
\begin{displaymath}
	\|u_n - u_m\|_{\alpha,p}\le
	\begin{cases}
		c_1^{\frac{1}{1-p}} \|A_p(u_n) - A_p(u_m)\|_{W^{-1,p'}_\alpha(\Omega)}^{\frac{1}{p-1}}  &\hbox{if $p>2$},\\
		c_2^{-1} M^{2-p} \|A_p(u_n) - A_p(u_m)\|_{W^{-1,p'}_\alpha(\Omega)} &\hbox{if $1<p\le 2$},
	\end{cases}
\end{displaymath}
so we deduce that $u_n \to u$ strongly in $\mathring{W}^{1,p}_\alpha(\Omega)$.
\end{proof}

\section{Existence of infinitely many solutions}
\label{sec_genus}

Observe that, due to the presence of the critical term, the functional $I_\lambda$ is neither coercive nor bounded from below. Therefore, we use a truncation argument to tackle the issue posed by the critical term. By the Sobolev embedding in Proposition \ref{prop_sobolev}, we get
\begin{equation}
\label{eq_est_Ilambda}
    I_\lambda(u)\ge \frac{1}{p} \int_\Omega |\nabla_\alpha u |^p\,\dz - \frac{1}{p^*_\alpha S^{p^*_\alpha/p}} \left(\int_\Omega |\nabla_\alpha u|^p \right)^{p^*_\alpha/p} - \frac{\lambda}{q} C_q \left( \int_\Omega |\nabla_\alpha u|^p  \right)^{q/p},
\end{equation}
for all $u\in\mathring{W}^{1,p}_\alpha(\Omega)$, where $S$ is defined in \eqref{eq:sobconst} and $C_q$ is the constant arising from the continuous embedding $\mathring{W}^{1,p}_\alpha(\Omega)\hookrightarrow L^q(\Omega)$, namely $C_q>0$ is such that
\begin{equation}
    \label{eq_cq}
    \|u\|_q\le C_q^{1/q}\|u\|_{\alpha,p} \quad \hbox{for all $u\in\mathring{W}^{1,p}_\alpha(\Omega)$.}
\end{equation}
Defining the function $g_\lambda\colon [0,+\infty)\to\R$  as
\begin{equation*}
    g_\lambda(x) \coloneqq \frac{1}{p} x^p - \frac{1}{p^*_\alpha S^{p^*_\alpha/p}} x^{p^*_\alpha} - \frac{\lambda}{q} C_q x^q
\end{equation*}
from \eqref{eq_est_Ilambda} we get
\begin{equation*}
I_\lambda(u) \ge g_\lambda(\|u\|_{\alpha,p}).    
\end{equation*}
Let us fix a sufficiently small $\hat{x}>0$ such that
\begin{equation*}
    \frac{1}{p} \hat{x}^p - \frac{1}{p^*_\alpha S^{p^*_\alpha/p}} \hat{x}^{p^*_\alpha}>0,
\end{equation*}
and choose $\lambda_0$ such that
\begin{equation}
\label{eq_est_lambda0}
    0<\lambda_0 < \frac{q}{C_q \hat{x}^q} \left(\frac{1}{p} \hat{x}^p - \frac{1}{p^*_\alpha S^{p^*_\alpha/p}} \hat{x}^{p^*_\alpha} \right).
\end{equation}
So, it holds $g_{\lambda_0}(\hat x) >0.$ Let us now define
\begin{equation}
\label{eq_def_x1}
  x_1\coloneqq \max\{x\in (0,\hat x)\, : g_{\lambda_0}(x)\le 0 \}.  
\end{equation}
Since $q<p$, we have that $g_{\lambda_0}(x)<0$ as $x$ approaches $0$. Therefore we can conclude that $g_{\lambda_0}(x_1) = 0$, as in Figure \ref{fig:grafico}.

\begin{figure}[h!]
\centering
\begin{tikzpicture}[scale=0.70]
\begin{axis}[
    axis lines=middle,
    xmin=0, xmax=3.2,
    ymin=0, ymax=0.7,
    xtick=\empty,
    ytick=\empty,
    axis line style={thick},
    domain=0:4,
    samples=200,
    clip=false
]
\node[blue, anchor=west] at (axis cs:2.1,0.55) {$g_{\lambda_0}$};
\addplot[thick, blue, domain=0:1] {-x^3 + x^4};
\addplot[thick, blue, domain=1:3.2] {-(1/2)*(x-2)^2 + 1/2};
\addplot[fill=blue] coordinates {(1,0)} circle (1pt);
\node[below] at (axis cs:1,0) {$x_1$};
\addplot[fill=blue] coordinates {(2,0)} circle (1pt);
\node[below] at (axis cs:2,0) {$\hat{x}$};
\end{axis}
\end{tikzpicture}
\caption{}
\label{fig:grafico}
\end{figure}

Let us now consider a smooth function $\Phi\colon [0,+\infty)\to [0,1]$ such that
\begin{equation*}
\Phi(x) =
\begin{cases}
1\!&\text{if } x \le x_1,\\
0\!&\text{if } x \ge \hat{x}
\end{cases}
\end{equation*}
and, for any $u\in\mathring{W}^{1,p}_\alpha(\Omega)$, we consider the truncated functional
\begin{equation*}
    J_\lambda(u) \coloneqq \frac{1}{p}\int_\Omega |\nabla_\alpha u|^p\,\dz - \frac{1}{p^*_\alpha S^{p^*_\alpha/p}} \left( \int_\Omega |u|^{p^*_\alpha}\,\dz \right) \Phi(\|u\|_{\alpha,p}) - \frac{\lambda}{q}\int_\Omega |u|^q\,\dz.
\end{equation*}
It is readily seen that $J_\lambda$ is bounded from below and coercive.
Other relevant properties of the $J_\lambda$ are listed below.
\begin{lemma}\label{lemma_lambda}
There exists $\bar\lambda$ such that for all $\lambda\in (0,\bar\lambda)$ the following hold.
    \begin{enumerate}[label=(\roman*)]
        \item If $J_\lambda(u)\le 0$ then $\|u\|_{\alpha,p}<x_1$ and $I_\lambda(u) = J_\lambda(u)$ in a neighborhood of $u$.
        \item The functional $J_\lambda$ satisfies the Palais-Smale condition for all $c<0$.
    \end{enumerate}
\end{lemma}
\begin{proof}
    \begin{enumerate}[label=(\roman*)]
        \item Let $J_\lambda(u)\le 0.$  We distinguish two cases. If $\|u\|_{\alpha,p}\ge \hat{x}$,
        \begin{equation*}
        J_\lambda(u)\ge \frac{1}{p}\|u\|_{\alpha,p}^p - \frac{\lambda C_q}{q} \|u\|_{\alpha,p} ^q = h_\lambda(\|u\|_{\alpha,p}),            
        \end{equation*}
        where $C_q$ is defined in \eqref{eq_cq} and $h_\lambda\colon [0,+\infty)\to\R$ is defined as
        \begin{equation*}
             h_\lambda(x) \coloneqq\frac{1}{p} x^p - \frac{\lambda C_q}{q} x^q.       
        \end{equation*}
        Observe that the nontrivial zero of $h_\lambda$ is given by $\tilde{x}= \left(\frac{\lambda p C_q}{q} \right)^{\frac{1}{p-q}}$ and, since $q<p$, $h_\lambda$ is positive and increasing for $x>\tilde{x}$. 
        If we choose $\lambda<\tilde{\lambda}$, where \begin{equation*}
    \tilde{\lambda} = \frac{q}{pC_q} \hat{x}^{p-q}
\end{equation*}
        then $\hat{x}>\tilde{x}$ and so 
        \begin{equation*}
        0\ge J_\lambda(u)\ge h_\lambda(\|u\|_{\alpha,p})> h_\lambda(\tilde{x}) = 0,            
        \end{equation*}
        which is a contradiction.

        If $\|u\|_{\alpha,p}<\hat{x}$, since $\Phi\le 1$, $J_\lambda(u) \ge g_\lambda(\|u\|_{\alpha,p})$ and so, choosing $\lambda<\lambda_0$, where $\lambda_0$ is defined in \eqref{eq_est_lambda0},
        \begin{equation*}
            0\ge J_\lambda(u)\ge g_\lambda(\|u\|_{\alpha,p}) > g_{\lambda_0}(\|u\|_{\alpha,p}),        
        \end{equation*}
        since the function $\lambda\mapsto g_\lambda(\cdot)$ is decreasing.
        Now, by the definition of $x_1$, it follows that $\|u\|_{\alpha,p}<x_1$ (see also again Figure \ref{fig:grafico}), that is $\Phi(\|u\|_{\alpha,p}) =1$ and $I_\lambda$ and $J_\lambda$ coincide. 
        Finally, by the continuity of the functional $J_\lambda$, there exists a neighborhood $U$ of $u$ such that $J_\lambda(u)<0$ for all $u\in U$. Therefore, for any $u\in U$, it follows that $I_\lambda(u) = J_\lambda(u)$.
        \item Let $c<0$ and $(u_n)\subset\mathring{W}^{1,p}_\alpha(\Omega)$ a sequence such that $J_\lambda(u_n)\to c$ and $J'_\lambda(u)\to 0$ as $n\to\infty$. By the previous statement we get that $I_\lambda (u_n) = J_\lambda (u_n)$ for sufficiently large $n\in\N$. Since $J_\lambda$ is coercive, the sequence $(u_n)$ is bounded in $\mathring{W}^{1,p}_\alpha(\Omega)$. 
        Choose $\lambda_*$ small such that
\begin{equation*}
    \frac{S^{N_\alpha/{p}}}{N_\alpha}- h \lambda_*^{ \frac{p^*_\alpha}{p^*_\alpha-q} }\ge 0,
\end{equation*}
where $h$ is defined in \eqref{eq_def_h}. For $\lambda<\lambda_*$, then $c< \frac{S^{N_\alpha/{p}}}{N_\alpha}- h \lambda_*^{ \frac{p^*_\alpha}{p^*_\alpha-q} }$ and Lemma \ref{lemma:Palais-Smale} gives that $J_\lambda$ satisfies the Palais-Smale condition at level $c$. 
    \end{enumerate}
Combining the previous results, the choice of $\bar\lambda = \min\{\tilde{\lambda}, \lambda_0, \lambda_*\}$ gives that (i) and (ii) are satisfied for all $\lambda\in (0,\bar\lambda)$.
\end{proof}

Before concluding the proof of Theorem \ref{thm_main}, we need some properties of the sublevels of the functional $J_\lambda$. Then we can construct the minimax levels associated to the functional.
We refer to Section \ref{sebsection_tools} for the details and the notation regarding Krasnoselskii's genus.

\medskip
For simplicity, for each $\varepsilon>0$ we set
\begin{equation*}
    J_\lambda^{-\varepsilon}\coloneqq\{ u\in\mathring{W}^{1,p}_\alpha(\Omega)\,: J_\lambda(u)\le -\varepsilon \}.
\end{equation*}
\begin{lemma}
\label{lemma_genus}
    For each $n\in\N$ there exists $\varepsilon = \varepsilon(n)$ such that
    \begin{equation*}
        \gamma\left( J_\lambda^{-\varepsilon}  \right) \ge n.
    \end{equation*}
\end{lemma}
\begin{proof}
Fix $n\in\N$, and let $E_n$ be a $n$-dimensional subspace of $\mathring{W}^{1,p}_\alpha(\Omega)$. So, $\|\cdot\|_{\alpha,p}$ and $\|\cdot\|_q$ are equivalent norms on $E_n$, that is there exists a constant $C_n>0$ such that
\begin{equation}
\label{eq_equivalence}
    C_n^{1/q}\|u\|_{\alpha,p} \le \|u\|_q,\quad \forall u\in E_n.
\end{equation}
Recall that $x_1$ was defined in \eqref{eq_def_x1}. For any $u\in E_n$ with $\|u\|_{\alpha,p}\le x_1 $  we get and from \eqref{eq_equivalence} that
\begin{equation}
\label{eq_aaa}
    I_\lambda(u) = J_\lambda(u) \le \frac{1}{p} \|u\|_{\alpha,p}^p - \frac{\lambda C_n}{q} \|u\|_{\alpha,p}^q.
\end{equation}
Let $S_\eta \coloneqq \{ u\in E_n : \|u\|_{\alpha,p} = \eta\}$, where $\eta$ is such that
\begin{equation}
\label{eq_eta}
    0<\eta<\min \left\{x_1, \left(\frac{\lambda C_n p}{q} \right)^{\frac{1}{p-q}} \right\}.
\end{equation}
So, by \eqref{eq_aaa}-\eqref{eq_eta}, for each $u\in S_\eta$
\begin{equation*}
    J_\lambda(u) \le \eta^q \left(\frac{1}{p} \eta^{p-q} - \frac{\lambda C_n}{q} \right) <0
\end{equation*}
and by the compactness of $S_\eta$ we can choose $\varepsilon = \varepsilon(n)>0$ such that $J_\lambda(u)<-\varepsilon$ for all $u\in S_\eta$. So we have $S_\eta \subset J_\lambda^{-\varepsilon}$ and by (i), (ii) and (v) of Proposition \ref{prop_genus} we get
\begin{equation*}
\gamma\left( J_\lambda^{-\varepsilon}  \right) \ge \gamma(S_\eta) = n,    
\end{equation*}
concluding the proof.
\end{proof}
For any $k\in\N$ we now define the min-max values $c_k$ as
\begin{equation*}
    c_k \coloneqq\inf_{A\in\mathcal{E}_k} \sup_{u\in A} J_\lambda(u),
\end{equation*}
where $\mathcal{E}_k = \{A\in\mathcal{E}\,: \gamma(A)\ge k\}$. Note that trivially follows that $c_k\le c_{k+1}$ for all $k\in\N$.

\begin{lemma}
\label{lemma_negative}
 $c_k\in (-\infty,0)$ for every $\lambda>0$ and $k\in\N$.
\end{lemma}

\begin{proof}
    Fix $\lambda>0$ and $k\in\mathbb{N}$. 
    Since $J_\lambda$ is bounded from below we have that $c_k>-\infty$. 
    
    Note that $J_\lambda^{-\varepsilon}$ is closed and symmetric,  moreover $0\notin J_\lambda^{-\varepsilon}$ since $J_\lambda(0) = 0$. From Lemma \ref{lemma_genus}, there exists $\varepsilon>0$ such that $\gamma(J_\lambda^{-\varepsilon})\ge k$, so $J_\lambda^{-\varepsilon}\in\mathcal{E}_k$ and $\sup_{u\in J_\lambda^{-\varepsilon}} J_\lambda(u)\le -\varepsilon$. Then
    \begin{equation*}
        c_k = \inf_{A\in\mathcal{E}_k}\sup_{u\in A} J_\lambda(u) \le \sup_{u\in J_\lambda^{-\varepsilon}} J_\lambda(u) \le -\varepsilon <0,
    \end{equation*}
    concluding the proof.
\end{proof}

\begin{lemma} \label{lemma_genus_2}
    Let $k\in\N$ and $\lambda\in (0,\bar\lambda)$, where $\bar\lambda$ is as in Lemma \ref{lemma_lambda}. If $c =c_k = c_{k+1} = \dots = c_{k+r}$ for some $ r\in\N$, then $\gamma(K_c)\ge r+1.$ 
\end{lemma}
\begin{proof}
    Let $\lambda\in (0,\bar\lambda)$ and $k\in\N$. By Lemma \ref{lemma_negative} it follows that $c = c_k = c_{k+1} = \dots c_{k+r} <0$. We can conclude, thanks to Lemma \ref{lemma_lambda}, that $J_\lambda$ satisfies the Palais-Smale condition at level $c$, and so the set $K_c$ is compact.

    Now, suppose that $\gamma(K_c)\le r$. Therefore, from Proposition \ref{prop_genus} (vi), there exists $\delta>0$ and a neighborhood $N_\delta(K_c)$ of $K_c$ such that $\gamma(N_\delta(K_c)) = \gamma(K_c)\le r$.
    
From (v) and (vii) of Lemma \ref{lemma_deformation} and from Remark \ref{remark_epsilon}, there exist $\varepsilon\in (0,-c)$
and an odd homeomorphism $\eta\colon \mathring{W}^{1,p}_\alpha(\Omega)\to \mathring{W}^{1,p}_\alpha(\Omega)$ such that
\begin{equation}
    \label{eq_omeom}
    \eta(J_\lambda^{c+\varepsilon}\setminus N_\delta(K_c))\subset J_\lambda^{c-\varepsilon}.
\end{equation}
    Now, by definition
    \begin{equation*}
    c = c_{k+r} = \inf_{A\in\mathcal{E}_{k+r}} \sup_{u\in A} J_\lambda(u)        
    \end{equation*}
and so there exists $A\in \mathcal{E}_{k+r}$ such that
    \begin{equation*}
        \sup_{u\in A} J_\lambda(u) < c+\varepsilon,
    \end{equation*}
    implying that $A\subset J_\lambda^{c+\varepsilon}$. Therefore, we get by \eqref{eq_omeom}
    \begin{equation}
        \label{eq_bound_genus}
        \eta(A\setminus N_\delta(K_c))\le \eta (J_\lambda^{c+\varepsilon}\setminus N_\delta(K_c))\subset J_\lambda^{c-\varepsilon}.
    \end{equation}
 On the other hand, using (i) and (iii) of Proposition \ref{prop_genus}, we obtain
  \begin{equation*}
      \gamma(\eta(\overline{A\setminus N_\delta(K_c)})) \ge \gamma(\overline{A\setminus N_\delta(K_c)})\ge \gamma(A) - \gamma(N_\delta(K_c))\ge k,
  \end{equation*}
  since $A\in\mathcal{E}_{k+r}$ and $\gamma(N_\delta(K_n))\le r$.
  
  Therefore, we have $\eta(\overline{A\setminus N_\delta(K_c)})\in\mathcal{E}_k$. From the definition of $c_k$, we get
  \begin{equation*}
      \sup_{u\in \eta(\overline{A\setminus N_\delta(K_c)})} J_\lambda(u)\ge c_k = c,
  \end{equation*}
  which is a contradiction with \eqref{eq_bound_genus}. Therefore, $\gamma(K_c)\ge r+1$.
\end{proof}

We can now conclude the proof of our main result.
\begin{proof}[Proof of Theorem \ref{thm_main}]
    Let $\lambda\in (0,\bar\lambda)$, where $\bar\lambda$ is as in Lemma \ref{lemma_lambda}. From Lemma \ref{lemma_negative}, it follows that $c_k<0$ and Lemma \ref{lemma_lambda} (ii) guarantees that $J_\lambda$ satisfies the Palais-Smale condition at level $c_k$. Therefore, we can conclude that $c_k$ is a critical value of $J_\lambda$ for every $k$, see e.g. \cite[Proposition 10.8]{AbrosettiMalchiodi}. Finally, by Lemma \ref{lemma_lambda} (i), the negative values $c_k$ are critical points of $I_\lambda$. 
    Now we distinguish two cases.
    \begin{enumerate}[label=(\roman*)]
        \item If $-\infty<c_1<c_2<\dots <c_k<c_{k+1}<\dots$ for every $k\in\N$, then $I_\lambda$ possesses infinitely many critical points.
        \item If there exist $k,r\in\N$ such that $c_{k} = c_{k+1} = \dots = c_{k+r} = c$, by Lemma \ref{lemma_genus_2} we get $\gamma(K_c)\ge r+1>1$. Then $K_c$ has infinitely many distinct points (see e.g. \cite[Remark 7.3]{R1986}) and by Lemma \ref{lemma_lambda} (ii), we obtain infinitely many negative critical values for $I_\lambda$. 
    \end{enumerate}
    Hence, in both cases, problem \eqref{problem:main} has infinitely many solutions for all $\lambda\in (0,\bar\lambda)$.
\end{proof}

\section*{Acknowledgments}
\noindent
The authors are members of \emph{Gruppo Nazionale per l’Analisi Matematica, la Probabilità e le loro Applicazioni} (GNAMPA) of the \emph{Istituto Nazionale di Alta Matematica} (INdAM). This work has been funded by the European Union - NextGenerationEU
within the framework of PNRR Mission 4 - Component 2 - Investment 1.1
under the Italian Ministry of University and Research (MUR) program
PRIN 2022 - grant number 2022BCFHN2 - Advanced theoretical aspects in
PDEs and their applications - CUP: H53D23001960006 and partially
supported by the INdAM-GNAMPA Research Project 2024: Aspetti
geometrici e analitici di alcuni problemi locali e non-locali in
mancanza di compattezza - CUP: E53C23001670001.

\bibliographystyle{plain}
\bibliography{refs}
   
\end{document}